\documentclass[11pt]{article}
\usepackage{amsmath,amssymb,amscd,color}
\newcommand{\ep}{{ \epsilon  }}
\newcommand{\del}{{ \delta  }}
\newcommand{\bq}{\begin{equation}}
\newcommand{\eq}{\end{equation}}
\newcommand{\pa}{\partial}
\newcommand{\R}{{ \mathbb{R}  }}

\newcommand{\bbr}{{ \mathbb{R}  }}

\newcommand{\calC}{{ \mathcal C  }}

\newcommand{\calK}{{ \mathcal K  }}
\newcommand{\calN}{{ \mathcal N  }}

\newcommand{\bke}[1]{\left( #1 \right)}

\newcommand{\norm}[1]{\left\Vert #1 \right\Vert}
\newcommand{\abs}[1]{\left| #1 \right|}

\newcommand{\om}{{ \omega  }}

\newcommand{\na}{\nabla}
\newcommand {\ga}{\gamma}
\newcommand {\al}{\alpha}

\newcommand{\Del}{\Delta}

\begin{document}
\bibliographystyle{plain}


\newtheorem{defn}{Definition}
\newtheorem{lemma}[defn]{Lemma}
\newtheorem{proposition}{Proposition}
\newtheorem{theorem}[defn]{Theorem}
\newtheorem{cor}{Corollary}
\newtheorem{remark}{Remark}
\numberwithin{equation}{section}

\def\Xint#1{\mathchoice
   {\XXint\displaystyle\textstyle{#1}}%
   {\XXint\textstyle\scriptstyle{#1}}%
   {\XXint\scriptstyle\scriptscriptstyle{#1}}%
   {\XXint\scriptscriptstyle\scriptscriptstyle{#1}}%
   \!\int}
\def\XXint#1#2#3{{\setbox0=\hbox{$#1{#2#3}{\int}$}
     \vcenter{\hbox{$#2#3$}}\kern-.5\wd0}}
\def\ddashint{\Xint=}
\def\dashint{\Xint-}
\def\aint{\Xint\diagup}

\newenvironment{proof}{{\bf Proof.}}{\hfill\fbox{}\par\vspace{.2cm}}
\newenvironment{pfthm1}{{\par\noindent\bf
            Proof of Theorem \ref{Theorem1} }}{\hfill\fbox{}\par\vspace{.2cm}}
\newenvironment{pfprop1}{{\par\noindent\bf
            Proof of Proposition  \ref{time-decay} }}{\hfill\fbox{}\par\vspace{.2cm}}
\newenvironment{pfthm3}{{\par\noindent\bf
            Proof of Theorem  \ref{Theorem2} }}{\hfill\fbox{}\par\vspace{.2cm}}

\newenvironment{pfthm4}{{\par\noindent\bf
Sketch of proof of Theorem \ref{Theorem6}.
}}{\hfill\fbox{}\par\vspace{.2cm}}
\newenvironment{pfthm5}{{\par\noindent\bf
Proof of Theorem 5. }}{\hfill\fbox{}\par\vspace{.2cm}}
\newenvironment{pflemsregular}{{\par\noindent\bf
            Proof of Lemma \ref{sregular}. }}{\hfill\fbox{}\par\vspace{.2cm}}

\title{Asymptotic behaviors of solutions for an aerobatic model coupled to fluid equations}
\author{Myeongju Chae, Kyungkeun Kang and Jihoon Lee}

\date{}

\maketitle
\begin{abstract}
We consider coupled system of Keller-Segel type equations and the
incompressible Navier-Stokes equations in spatial dimension two. We
show temporal decay estimates of solutions with small initial data
and obtain their asymptotic profiles as time tends to infinity.
\newline{\bf 2000 AMS Subject
Classification}: 35Q30, 35Q35, 76Dxx, 76Bxx
\newline {\bf Keywords}: asymptotic behavior, Keller-Segel,
Navier-Stokes equations
\end{abstract}

\section{Introduction}
 \setcounter{equation}{0}
In this paper, we consider a mathematical model  describing the
dynamics of  oxygen, swimming bacteria, and viscous incompressible
fluids in $\bbr^2$.
\begin{equation}\label{KSNS} \left\{
\begin{array}{ll}
\partial_t n + u \cdot \nabla  n - \Delta n= -\nabla\cdot (\chi (c) n \nabla c),\\
\vspace{-3mm}\\
\partial_t c + u \cdot \nabla c-\Delta c =-k(c) n,\\
\vspace{-3mm}\\
\partial_t u + u\cdot \nabla u -\Delta u +\nabla p=-n \nabla
\phi,\quad
\nabla \cdot u=0 
\end{array}
\right. \quad\mbox{ in }\,\, Q_{T} := (0,\, T) \times \R^{2},
\end{equation}
where $c(t,\,x) : Q_{T} \rightarrow \R^{+}$, $n(t,\,x) : Q_{T}
\rightarrow \R^{+}$, $u(t,\, x) : Q_{T} \rightarrow \R^{d}$ and
$p(t,x) :  Q_{T} \rightarrow \R$ denote the oxygen concentration,
cell concentration, fluid velocity, and scalar pressure,
respectively. Here $\R^+$ indicates the set of non-negative real
numbers.
Such a model was proposed by Tuval et al.\cite{TCDWKG}, formulating
the dynamics of swimming bacteria, {\it Bacillus subtilis} (see
\cite{TCDWKG} for more details on biological phenomena).
%

The nonnegative functions $k(c)$ and $\chi (c)$ denote the oxygen
consumption rate and the aerobatic sensitivity, respectively, i.e.
$k, \chi:\R^+\rightarrow\R^+$ such that $k(c)=k(c(x,t))$ and
$\chi(c)=\chi(c(x,t))$. Initial data are given by $(n_0(x), c_0(x),
u_0(x))$ with $n_0(x),\, c_0(x) \geq 0$ and $\nabla \cdot u_0=0$.
To describe the fluid motions, Boussinesq approximation is used to
denote the effect due to heavy bacteria. The time-independent
function $\phi =\phi (x)$ denotes the potential function produced by
different physical mechanisms, e.g., the gravitational force or
centrifugal force.
\\
\indent We can compare the above system \eqref{KSNS} to the
classical Keller-Segel model, suggested by Patlak\cite{Patlak} and
Keller-Segel\cite{KS1, KS2}, which is given as
\begin{equation}\label{KS-nD} \,\,\left\{
 \begin{array}{c}
 n_t=\Delta n-\nabla \cdot(n \chi\nabla c),\\
 \vspace{-3mm}\\
  c_t=\Delta c-\alpha c+\beta n,\\
 \end{array}
 \right.
\end{equation}
where $n=n(t,x)$ is the cell density and $c=c(t,x)$ is the
concentration of chemical attractant substance. Here, $\chi$ is the
chemotatic sensitivity, and $\alpha$ and $\beta$ are the decay and
production rate of the chemical, respectively. The system
\eqref{KS-nD} has been comprehensively studied and we will not try
to give list of results here (see e.g. \cite{Her-Vela, NSY, OY, Win}
and the survey papers \cite{Horstmann1, Horstmann2}). In the absence
of effect of fluids, i.e., $u=0$, the system \eqref{KSNS} has some
similarities to the Keller-Segel equations \eqref{KS-nD} and
however, we emphasize that the oxygen concentration in \eqref{KSNS}
is consumed and the chemical substance, meanwhile, is produced by
$n$ in \eqref{KS-nD}. That's why the righthand side of the second
equation in \eqref{KSNS} or \eqref{KS-nD} has a different sign.


We review some known results related to our concerns. In \cite{Lorz}
existence of solutions was shown locally in time for bounded domains
in $\R^3$ and \cite{DLM} proved that smooth solutions are globally
extended in time if initial data are sufficiently close to constant
steady states and if $\chi(\cdot), k(\cdot)$ satisfy the following
conditions:
\begin{equation}\label{Assumption1}
\chi'(\cdot)\ge 0,\quad k'(\cdot) >0,\quad
\left(\frac{k(\cdot)}{\chi(\cdot)} \right)^{''} <0.
\end{equation}
It was also shown in \cite{DLM} that weak solutions exist globally
in time in $\R^2$, provided that the initial chemical concentration
is small.
In $\R^2$, \cite{Win2} proved the global existence of regular
solutions without smallness assumptions on initial data for bounded
domains with boundary conditions $\partial_{\nu}n=\partial_{\nu}
c=u=0$ under the following sign conditions on $\chi(\cdot)$ and
$k(\cdot)$:
\begin{equation}\label{Assumption2}
\left( \frac{k(\cdot)}{\chi(\cdot)} \right)^{'} >0, \quad
(\chi(\cdot) k(\cdot))' \ge 0, \quad \left(
\frac{k(\cdot)}{\chi(\cdot)} \right) ^{''} \le 0.
\end{equation}
In \cite{ckl} the authors of the paper established global existence
of smooth solutions in $\R^2$ with no smallness of the initial data
and certain conditions, motivated by experimental results in
\cite{CFKLM} and \cite{TCDWKG}, on $\chi(\cdot)$ and $k(\cdot)$
(compare to \eqref{Assumption2}), that is,
\begin{equation} \label{CKL10-march14}
\chi(c),\, k(c),\, \chi'(c),\, k'(c)\geq 0,\,\mbox {and }\, \sup |
\chi (c)- \mu k(c)| < \epsilon\,\,\mbox{ for some }\,\mu>0.
\end{equation}
Construction of weak solutions in $\R^3$ was also established in
\cite{ckl} in case that $\abs{\chi (c)- \mu k(c)}=0$ in
\eqref{CKL10-march14}. The authors also studied the time decay of
regular solution in \cite{ckl-cpde}. More precisely, it was shown
that if $L^{\infty}$-norm of $c_0$ is sufficiently small, then
regular solution exists globally and, furthermore, $n$ and $c$
satisfy the following time decay:
\begin{equation}\label{CK30L-march14}
\| n(t)\|_{L^{\infty}(\R^d)} + \|c(t)\|_{L^{\infty}(\R^d)} \le
C(1+t)^{-\frac d4},\qquad d=2, 3.
\end{equation}
For bounded convex domains with smooth boundary, \cite{Win3} showed
that $(n, c, u)$ converges to $((n)_a, 0, 0)$ in $L^{\infty}$-norm
under the assumption \eqref{Assumption2}, where $(n)_a$ indicates
the mean value of $n_0$. We consult \cite{CKK}, \cite{FLM} and
\cite{Tao-W} with reference therein for the nonlinear diffusion
models of a porous medium type.


Our main objective of this paper is to obtain asymptotic profiles of
temporal decaying solutions of \eqref{KSNS}. To be more precise, if
certain norms of initial data are sufficiently small, we prove
existence of global regular solutions, which show certain degree of
temporal decay, and in additions, asymptotic profiles of $n$ and $u$
can be obtained.

Before we state our main result, since the vorticity equation is
rather convenient than the equation of velocity, we consider from
now on
\begin{equation}\label{CKL20-dec12}
\partial_t n + u \cdot \nabla  n - \Delta n= -\nabla\cdot (\chi (c) n \nabla c),
\end{equation}
\begin{equation}\label{CKL30-dec12}
\partial_t c + u \cdot \nabla c-\Delta c =-k(c) n,
\end{equation}
\begin{equation}\label{CKL40-dec12}
\partial_t \omega + u\cdot \nabla \omega-\Delta \omega=-\nabla^{\perp}(n \nabla
\phi),
\end{equation}
where $u$ is given as a Biot-Savart law, namely
\begin{equation}\label{CKL20-jan9}
u=K*\omega,\qquad K(x)=\nabla^{\perp} \log
|x|=<-\frac{x_2}{\abs{x}^2}, \frac{x_1}{\abs{x}^2}>.
\end{equation}
We denote by $m$ and $\gamma$ the total mass of $n$ and total
circulation of $\omega$, respectively, i.e.
\begin{equation}\label{CKL101-dec12}
\int_{\R^2} n_0(x)dx=m, \qquad \int_{\R^2} \omega_0(x)dx=\gamma
\end{equation}

We are ready to sate our main result, which reads as follows:

\begin{theorem}\label{Theorem1}
Let the initial data $(n_0, c_0, u_0)$ be given in
$H^{m-1}(\bbr^d)\times H^m (\bbr^d)\times H^m(\bbr^d)$ for $m\geq 3$
and $d=2$ with $n_0\geq 0$ and $c_0\geq 0$. Assume that $\chi, k,
\chi', k'$ are all non-negative and
 $\chi$, $k\in C^m(\R^+)$ and $k(0)=0$, $\|
\nabla^l \phi \|_{L^1\cap L^{\infty}}<\infty$ for $1\le |l|\le m$.
There exists a constant $ \epsilon_1>0 $ such that if
\begin{equation}\label{CKL-assumption10}
\norm{n_0}_{L^1(\R^2)}+\norm{c_0}_{L^{\infty}(\R^2)}
+\norm{\omega_0}_{L^1(\R^2)}
<\epsilon_1,
\end{equation}
then unique classical solutions $(n, c, \omega)$ of
\eqref{CKL20-dec12}-\eqref{CKL20-jan9} exist globally and  $(n, c,
\omega)$ satisfy the following asymptotics: for any $R<\infty$ and
for all $1<r<\infty$
\[
\lim_{t\rightarrow\infty}t\norm{n(\cdot,
t)-m\Gamma(\cdot,t)}_{L^{\infty}(B_{t,R})}=0,
\]
\[
\lim_{t\rightarrow\infty}t^{\frac{1}{2}}\norm{\nabla c(\cdot,
t)}_{L^{\infty}(B_{t,R})}=0,
\]
\[
\lim_{t\rightarrow\infty}t^{1-\frac{1}{r}}\norm{\omega(\cdot,
t)-\gamma\Gamma(\cdot,t)}_{L^{r}(B_{t,R})}=0,
\]
where $B_{t,R}:=\{x\in\R^2: \abs{x}<Rt^{\frac{1}{2}}\}$ and
$\Gamma(x,t)$ is  the two dimensional heat kernel, i.e.
$\Gamma(x,t)=(4\pi t)^{-1}\exp(-\abs{x}^2/4t)$.
\end{theorem}
\begin{remark}
The unique existence of  classical solution was proved previously in
\cite{ckl-cpde} assuming either $\norm{n_0}_{L^1(\R^2)}<\epsilon_1$
or $\|c_0\|_{L^{\infty}} < \ep_1 $. The smallness condition of
\eqref{CKL-assumption10}
 is necessary to obtain the
time decay and asymptotic behaviors.
We also note that Theorem \ref{Theorem1} implies
the following temporal decay of $(n, c, \omega)$ for large $t$:
\[
\| n(t)\|_{L^{\infty}(\R^2)} \sim\frac{m}{t}+\frac{o(1)}{t}, \qquad
\|\nabla c(t)\|_{L^{\infty}(\R^2)} \sim
\frac{o(1)}{t^{\frac{1}{2}}},
\]
\[
\|
\omega(t)\|_{L^{r}(\R^2)}\sim\frac{\gamma}{t^{1-\frac{1}{r}}}
+\frac{o(1)}{t^{1-\frac{1}{r}}},\qquad
1<r<\infty.
\]
\end{remark}

This paper is organized as follows.
Section 2 is devoted to obtaining decay rate of solutions in case
that certain norm of initial data are sufficiently small. In Section
3, we present the proof of Theorem \ref{Theorem1}.



\section{Estimates of temporal decay}

We first introduce the notation and present preparatory results that
are useful to our analysis. We start with the notation. For $1\leq
q\leq \infty$, we denote by $W^{k,q}(\Omega)$ the usual Sobolev
spaces, namely $W^{k,q}(\Omega)=\{f\in L^q(\Omega): D^{\alpha}f\in
L^q(\Omega), 0\leq \abs{\alpha}\leq k\}$. The letter $C$ is used to
represent a generic constant, which may change from line to line,
and $C(*,\cdots,*)$ is considered a positive constant depending on
$*,\cdots,*$. Sometimes, we use $A\lesssim B$, which means the
inequality $A\le CB$, where $C$ is a generic constant. For
convenience we mention the elementary inequalities which are
repeatedly used;
\begin{equation}\label{CKL-Jan14-10}
\int_0^t \frac{1}{(t-s)^{1-a}} \frac{1}{s^{1-b}} ds \le
\frac{C}{t^{1-(a+b)}} \quad (a>0,  b>0)
\end{equation}
\begin{equation}\label{CKL-Jan14-20}
\int_0^{\frac t2} \frac{1}{(t-s)^b} \frac{1}{s^{1-a}} ds  \le
\frac{C}{t^{b-a}}\qquad \quad\int_{\frac t2}^t \frac{1}{(t-s)^{1-a}}
\frac{1}{s^{b}} ds  \le \frac{C}{t^{b-a}} \quad (a>0,  b\ge 0).
\end{equation}
We remind a lemma in \cite[section 2.2.5]{GGS} and the following is
its slight modified version.
\begin{lemma}\label{CKL-GGS-100}
Let $f:\R^2\rightarrow\R$ and $g:\R^2\rightarrow\R^2$ be $\calC^1$
and radial in $\R^2$. Then,
\[
\bke{(K*g) \nabla }f=0 \qquad \mbox{ in }\,\,\R^2,
\]
where $K(x)=<-\frac{x_2}{\abs{x}^2}, \frac{x_1}{\abs{x}^2}>$.
\end{lemma}
\begin{proof}
The proof can be similarly proved by the same arguments as the Lemma
in \cite[section 2.2.5]{GGS}, and therefore, we skip its details.
\end{proof}

In this section, we are concerned with optimal temporal decays of
solutions $(n, c, \omega)$ of
\eqref{CKL20-dec12}-\eqref{CKL20-jan9}, and our main goal is to
prove the next proposition. Let us recall the smallness assumption in Theorem \ref{Theorem1}:
\begin{equation}\label{Nov17-CKL30}
\norm{n_0}_{L^1(\R^2)}+\norm{c_0}_{L^{\infty}(\R^2)}
+\norm{\omega_0}_{L^1(\R^2)}<\epsilon_1,
\end{equation}
where $\omega_0=\nabla\times u_0$.
\begin{proposition}\label{time-decay}
Assume the condition of Theorem \ref{Theorem1} holds. The classical solutions $(n, c, \omega)$ of
\eqref{CKL20-dec12}-\eqref{CKL20-jan9} exist globally and $(n, c,
\omega)$ satisfy the following time decay:
\begin{equation}\label{CKL10-june14}
\| n(t)\|_{L^{\infty}(\R^2)} \le \frac{C\epsilon_1}{t},\qquad \|
\nabla n(t)\|_{L^{\infty}(\R^2)} \le
\frac{C\epsilon_1}{t^{\frac{3}{2}}},
\end{equation}
\begin{equation}\label{CKL10-jan9}
\|\nabla c(t)\|_{L^{\infty}(\R^2)} \le
\frac{C\epsilon_1}{t^{\frac{1}{2}}},\qquad \|\nabla^2
c(t)\|_{L^{\infty}(\R^2)} \le \frac{C\epsilon_1}{t},
\end{equation}
\begin{equation}\label{CKL10-dec12}
\| \omega(t)\|_{L^{r}(\R^2)} \le \frac{C\epsilon_1}{t^{1-\frac{1}{r}}} \quad 1<r<\infty, \qquad
\| \na \omega(t)\|_{L^r(\R^2)}\le \frac{C\epsilon_1}{t^{\frac 32 -\frac{1}{r}}}\quad1\le r< 2.
\end{equation}
\end{proposition}

The proof of Proposition \ref{time-decay} will be presented in the
series of lemmas. Lemma \ref{lemma-higher-est} considers the decays
of $\| n\|_{L^{\infty}}(t), \| \na c\|_{L^{\infty}}(t), \|
\om\|_{L^r}(t)$, and Lemma \ref{CKL-Dec05-100} shows the decays of
quantities
with derivatives. Notice that the decay rates in
\eqref{CKL10-june14} and \eqref{CKL10-dec12} are the same as in the
$L^q-L^1$ estimate for the two dimensional heat equation. In this
regard our approach is to see the system \eqref{KSNS} as the
perturbed heat equations with the smallness assumption
\eqref{CKL-assumption10}, and to apply the linear heat kernel
estimates
\begin{equation}\label{heatK}
\| \na ^{\al} e^{- \Del t }u \|_{L^q(\R^2)}  \le
C t^{- (1/r-1/q) - |\al|/2} \| u\|_{L^r(\bbr^2)}, \qquad 1 \le r\le
q \le \infty.
\end{equation}
In doing so, we need an intermediate step (Lemma \ref{lemma-ncw-100}
shown below), which establishes $(n, \na c, \om)$ to be small in a
weighted norms in time variable (Lemma \ref{lemma-higher-est} and
Lemma \ref{CKL-Dec05-100} shown below). This types of estimates for
weighted norms can be found in \cite{Na}. Due to Lemma
\ref{lemma-ncw-100} we work out Lemma \ref{lemma-higher-est} and
Lemma
 \ref{CKL-Dec05-100} so that the nonlinear terms in the Duhamel's formula  are estimated by either
 quadratic terms or terms multiplied with small parameter $\ep_1$ (see e.g. \eqref{Dec03-CKL150}).\\
%
%
Let us introduce some spaces of
functions defined as follows:
\begin{equation}\label{Nov17-CKL40}
\norm{n}_{\calK_p(\R^2)}:=\sup_{t\ge
0}t^{1-\frac{1}{p}}\norm{n(t)}_{L^p(\R^2)},
\end{equation}
\begin{equation}\label{Nov17-CKL50}
\norm{c}_{\calN_q(\R^2)}:= \sup_{t\ge
0}t^{\frac{1}{2}-\frac{1}{q}}\norm{\nabla c(t)}_{L^q(\R^2)},
\end{equation}
\begin{equation}\label{Nov17-CKL60}
\norm{\omega}_{\calK_r(\R^2)}:=\sup_{t\ge
0}t^{1-\frac{1}{r}}\norm{\omega(t)}_{L^r(\R^2)}.
\end{equation}
For convenience,
we denote
\[
\norm{(n, c,
u)}_{\calK_{p,q,r}}:=\norm{n}_{\calK_p}+\norm{c}_{\calN_q}+\norm{\omega}_{\calK_r}.
\]

\begin{lemma}\label{lemma-ncw-100}
Let $n, c$ and $\omega$ be solutions of
\eqref{CKL20-dec12}-\eqref{CKL20-jan9}. Suppose that the assumptions
in Theorem \ref{Theorem1} are satisfied, and $p, q, r$ are in the
range of
\begin{equation}\label{Nov17-CKL20}
\frac{4}{3}<p<2,\qquad 2<q<4,\qquad 1<r<2.
\end{equation}
Then, we have
\begin{equation}\label{Dec1-CKL10}
\norm{(n,c,\omega)}_{\calK_{p,q,r}} \le
C(\norm{n_0}_{L^1}+\norm{c_0}_{L^{\infty}}
+\norm{\omega_0}_{L^1})\le C\epsilon_1.
\end{equation}
\end{lemma}
\begin{proof}
First, we write the equations as integral representation.
\begin{equation}\label{Nov28-CKL80}
n(t)=e^{t{\Delta}}n_0+\int_0^t\nabla
e^{(t-s){\Delta}}\bke{\chi(c)n(s)\nabla c(s)}ds+\int_0^t\nabla
e^{(t-s){\Delta}}\bke{u(s)n(s)}ds ,
\end{equation}
\begin{equation}\label{Nov28-CKL90}
c(t)=e^{t{\Delta}}c_0-\int_0^t
e^{(t-s){\Delta}}\bke{k(c)n(s)}ds-\int_0^t
e^{(t-s){\Delta}}\bke{u(s)\nabla c(s)}ds,
\end{equation}
\begin{equation}\label{Nov28-CKL100}
\omega(t)=e^{t{\Delta}}\omega_0+\int_0^t
\nabla^{\perp}e^{(t-s){\Delta}}\bke{n(s)\nabla\phi}ds+\int_0^t
\nabla e^{(t-s){\Delta}}\bke{u(s) \omega(s)}ds,
\end{equation}
where $\nabla^{\perp}=(-\partial_{x_2}, \partial_{x_1})$. Using the
estimate of the heat kernel, we obtain
\[
\| n(t) \|_{L^p} \lesssim t^{-1+\frac{1}{p}} \norm{n_0 }_{L^1}
+\int_0^t \norm{ \nabla e^{(t-s)\Delta}}_{L^{\alpha}}
\norm{n(s)}_{L^p} \norm{\nabla c(s)}_{L^q}ds
\]
\begin{equation}\label{Nov28-00}
+\int_0^t \norm{\nabla e^{(t-s)\Delta}}_{L^{\alpha'}}\norm{ u(s)
}_{L^{\frac{2r}{2-r}}}\norm{n (s) }_{L^p} ds=t^{-1+\frac{1}{p}}
\norm{n_0 }_{L^1}+I_1+I_2,
\end{equation}
where $1+\frac{1}{p}=\frac{1}{\alpha}+\frac{1}{p}+\frac{1}{q}$ and
$\frac32-\frac{1}{r}=\frac{1}{\alpha'}$. We estimate $I_1$ and $I_2$
as follows:
\[
I_1 \lesssim \int_0^t
\frac{1}{(t-s)^{\frac32-\frac{1}{\alpha}}}\cdot
\frac{1}{s^{\frac32-\frac{1}{p}-\frac{1}{q}}}
ds\norm{n}_{\calK_p}\norm{c}_{\calN_r}
\lesssim
\frac{1}{t^{1-\frac{1}{p}}}\norm{n}_{\calK_p}\norm{c}_{\calN_r},
\]
\[
I_2 \lesssim \int_0^t
\frac{1}{(t-s)^{\frac32-\frac{1}{\alpha'}}}\cdot
\frac{1}{s^{2-\frac{1}{r}-\frac{1}{p}}} ds \norm{\omega
}_{\calK_r}\norm{n}_{\calK_p}
\lesssim \frac{1}{t^{1-\frac{1}{p}}} \norm{\omega
}_{\calK_r}\norm{n}_{\calK_p},
\]
where we used \eqref{CKL-Jan14-10}. Therefore, we obtain
\begin{equation}\label{CKL-Jan12-20}
\norm{n}_{\calK_p}\leq C\norm{n_0
}_{L^1}+C\norm{n}_{\calK_p}\norm{c}_{\calN_r}+C\norm{\omega
}_{\calK_r}\norm{n}_{\calK_p}.
\end{equation}
Similarly, we obtain
\[
\norm{c}_{\calN_q} \le C\norm{ c_0}_{L^\infty} +C\sup |k(c)|
\norm{n}_{\calK_p} +C\norm{c}_{\calN_q}\norm{\omega}_{\calK_r}
\]
\begin{equation}\label{CKL-Jan12-30}
\le C\norm{ c_0}_{L^\infty} +C\norm{k(c)}_{L^{\infty}}
\norm{n}_{\calK_p} +C\norm{c}_{\calN_q}\norm{\omega}_{\calK_r}.
\end{equation}
Next, we estimate the vorticity.
\[
\norm{\omega (t)}_{L^r} \lesssim t^{-1+\frac{1}{r}}
\norm{\omega_0}_{L^1} +\int_0^t \norm{\nabla
e^{(t-s)\Delta}}_{L^{\alpha}}\norm{n(s)}_{L^p} \norm{\nabla \phi
}_{L^{2}}
\]
\[
+ \int_0^t \norm{\nabla e^{(t-s)\Delta}}_{L^{\alpha'}}
\norm{u}_{L^{\frac{2r}{2-r}}} \norm{\omega}_{L^r}
ds=t^{-1+\frac{1}{r}} \norm{\omega_0}_{L^1}+J_1+J_2,
\]
where $\frac{1}{r}=\frac{1}{\alpha} +\frac{1}{p} -\frac{1}{2}$ and
$\frac{1}{\alpha'}=\frac32-\frac{1}{r}$. Similar estimates as above
yield
\[
J_1 \lesssim \int_0^t \frac{1}{(t-s)^{\frac32-\frac{1}{\alpha}}}
\frac{1}{s^{1-\frac{1}{p}}} ds \norm{\nabla \phi
}_{L^{2}}\norm{n}_{\calK_p} \lesssim
\frac{1}{t^{1-\frac{1}{r}}}\norm{\nabla \phi
}_{L^{2}}\norm{n}_{\calK_p}.
\]
On the other hand, via $\norm{u(t)}_{L^s}\lesssim
\norm{\omega(t)}_{L^r}$ with $1/r=1/s+1/2$, we obtain
\[
J_2 \lesssim \int_0^t \frac{1}{(t-s)^{\frac32-\frac{1}{\alpha'}}}
\frac{1}{s^{2(1-\frac{1}{r})}} ds \norm{\omega }_{\calK_r}^2
\lesssim\frac{1}{t^{1-\frac{1}{r}}}\norm{\omega }_{\calK_r}^2.
\]
Thus, we have
\begin{equation}\label{CKL-Jan12-40}
\norm{\omega }_{\calK_r}\le C\norm{\omega_0}_{L^1} +C\norm{\nabla
\phi }_{L^{2}}\norm{n}_{\calK_p}+C\norm{\omega }_{\calK_r}^2.
\end{equation}
Here we set $M_1:=C\norm{k(c)}_{L^{\infty}}$ and $M_2:=C\norm{\nabla
\phi }_{L^{2}}$, where $C$ are the constants in \eqref{CKL-Jan12-30}
and \eqref{CKL-Jan12-40}. Multiplying \eqref{CKL-Jan12-20} with
$2(M_1+M_2)$ and summing up the above estimates,
\[
(M_1+M_2)\norm{n}_{\calK_p}+\norm{c}_{\calN_q} +\norm{\omega
}_{\calK_r}\leq
C(2(M_1+M_2)\norm{n_0}_{L^1}+\norm{c_0}_{L^{\infty}}+\norm{\omega_0}_{L^1})
\]
\begin{equation}\label{CKL-Jan12-50}
+2C(M_1+M_2)\norm{n}_{\calK_p}\norm{c}_{\calN_r}+2C(M_1+M_2)\norm{\omega
}_{\calK_r}\norm{n}_{\calK_p}
+C\norm{c}_{\calN_q}\norm{\omega}_{\calK_r}+C\norm{\omega
}_{\calK_r}^2.
\end{equation}
Therefore, we obtain
\begin{equation}\label{CKL-Jan12-60}
\norm{(n,c,\omega)}_{\calK_{p,q,r}} \le
C(\norm{n_0}_{L^1}+\norm{c_0}_{L^{\infty}}+\norm{\omega_0}_{L^1})
+C\norm{(n,c,\omega)}_{\calK_{p,q,r}}^2.
\end{equation}
We deduce the lemma by the standard theory of local well-posedness
argument.
\end{proof}

Next we show the decay of $(n, c, \omega)$ in $L^{\infty}\times
L^{\infty}\times L^{r}$ for $2\le r <\infty$.

\begin{lemma}\label{lemma-higher-est}
Let $n, c$ and $\omega$ be solutions of
\eqref{CKL20-dec12}-\eqref{CKL20-jan9}. If the assumptions in
Theorem \ref{Theorem1} are satisfied, then
\begin{align}
\label{Dec3-CKL10}
\norm{n(t)}_{L^{\infty}(\R^2)}\leq \frac{C\epsilon_1}{t},\qquad
\norm{\nabla c(t)}_{L^{\infty}(\R^2)}\leq
\frac{C\epsilon_1}{t^{\frac{1}{2}}},\\
 \label{Dec3-CKL20}
\norm{\omega(t)}_{L^{r}(\R^2)}\leq
\frac{C\epsilon_1}{t^{1-\frac{1}{r}}},\qquad 2\le r<\infty.
\end{align}
\end{lemma}

\begin{proof}
For convenience, we denote
\begin{equation*}
\norm{n}_{\calK_{\infty}(\R^2)}:=\sup_{t\ge
0}t\norm{n(t)}_{L^{\infty}(\R^2)}, \qquad\norm{
c}_{\calN_{\infty}(\R^2)}:= \sup_{t\ge 0}t^{\frac{1}{2}}\norm{\nabla
c(t)}_{L^{\infty}(\R^2)},
\end{equation*}
\begin{equation*}
\norm{\omega}_{\calK_r(\R^2)}:=\sup_{t\ge
0}t^{1-\frac{1}{r}}\norm{\omega(t)}_{L^r(\R^2)},\qquad 1< r<\infty.
\end{equation*}
Using the estimate of heat kernel, we obtain
\[
\norm{n}_{L^{\infty}}(t)\lesssim \frac{1}{t}\norm{n_0}_{L^1}
+\int_0^t\norm{\nabla e^{(t-s)\Delta} n\nabla c}_{L^{\infty}}(s)ds
\]
\[
+\int_0^t\norm{\nabla e^{(t-s)\Delta} un}_{L^{\infty}}(s) ds=
\frac{1}{t}\norm{n_0}_{L^1}+I_1+I_2.
\]
We first estimate $I_1$.
\[
I_1\lesssim \int_0^{t/2}\frac{1}{(t-s)^{\frac{3}{2}}}\norm{n\nabla
c}_{L^1}(s)ds
+\int_{t/2}^t\frac{1}{(t-s)^{\frac{1}{2}}}\norm{n\nabla
c}_{L^{\infty}}(s)ds
\]
\[
\lesssim
\int_0^{t/2}\frac{1}{(t-s)^{\frac{3}{2}}}\norm{n}_{L^1}\norm{\nabla
c}_{L^{\infty}}ds
+\int_{t/2}^t\frac{1}{(t-s)^{\frac{1}{2}}}\norm{n}_{L^{\infty}}\norm{\nabla
c}_{L^{\infty}}ds
\]
\begin{equation}\label{Dec03-CKL120}
\lesssim 
\frac{\epsilon_1}{t}\norm{\nabla c}_{\calN_{\infty}(\R^2)}
+\frac{1}{t}\norm{n}_{\calK_{\infty}(\R^2)}\norm{
c}_{\calN_{\infty}(\R^2)},
\end{equation}
where we used \eqref{CKL-Jan14-20}. For convenience, we introduce
H\"{o}lder conjugate numbers $2^+$ and $2^-$ so that
\[1/2^+=1/2-1/\alpha, \quad
1/2^-=1/2+1/\alpha, \quad 2<\alpha<\infty.\]  We then estimate $I_2$
as follows:
\[
I_2\lesssim
\int_0^{t/2}\frac{1}{(t-s)^{\frac{3}{2}}}\norm{un}_{L^1}(s)ds
+\int_{t/2}^t\frac{1}{(t-s)^{\frac{3}{2}-\frac{1}{2^{-}}}}\norm{un}_{L^{2^{+}}}(s)ds
\]
\[
\lesssim
\int_0^{t/2}\frac{1}{(t-s)^{\frac{3}{2}}}\norm{u}_{L^{2^+}}\norm{n}_{L^{2^-}}ds
+\int_{t/2}^t\frac{1}{(t-s)^{\frac{3}{2}-\frac{1}{2^{-}}}}\norm{u}_{L^{2^+}}\norm{n}_{L^{\infty}}ds
\]
\[
\lesssim
\int_0^{t/2}\frac{1}{(t-s)^{\frac{3}{2}}}\norm{u}_{L^{2^+}}\norm{n}_{L^{2^-}}ds
+\int_{t/2}^t\frac{1}{(t-s)^{\frac{3}{2}-\frac{1}{2^{-}}}}\norm{u}_{L^{2^+}}\norm{n}_{L^{\infty}}ds
\]
\[
\lesssim
\frac{1}{t^{\frac{3}{2}}}\int_0^{t/2}\norm{\omega}_{L^{\frac{\alpha}{\alpha-1}}}\norm{n}_{L^{2^-}}ds
+\int_{t/2}^t\frac{1}{(t-s)^{\frac{3}{2}-\frac{1}{2^{-}}}}\norm{\omega}_{L^{\frac{\alpha}{\alpha-1}}}\norm{n}_{L^{\infty}}ds
\]
\begin{equation}\label{Dec31-CKL10}
\lesssim\frac{1}{t}\norm{n}_{\calK_{2^-}(\R^2)}\norm{\omega}_{\calK_{\frac{\alpha}{\alpha-1}}(\R^2)}+
\frac{1}{t}\norm{\omega}_{\calK_{\frac{\alpha}{\alpha-1}}(\R^2)}\norm{n}_{\calK_{\infty}(\R^2)}
\lesssim\frac{\epsilon^2_1}{t}+\frac{\epsilon_1}{t}\norm{n}_{\calK_{\infty}(\R^2)},
\end{equation}
where we used the result in Lemma \ref{lemma-ncw-100}. Adding the
estimates, we obtain
\begin{equation}\label{Dec31-CKL20}
\norm{n}_{L^{\infty}}(t)\lesssim
\frac{\epsilon_1}{t}+\frac{\epsilon_1}{t}\norm{n}_{\calK_{\infty}(\R^2)}
+\frac{\epsilon_1}{t}\norm{c}_{\calN_{\infty}(\R^2)}
+\frac{1}{t}\norm{n}_{\calK_{\infty}(\R^2)}\norm{
c}_{\calN_{\infty}(\R^2)}.
\end{equation}
On the other hand, $\nabla c$ is computed as follows:
\[
\norm{\nabla c}(t)\lesssim
\frac{1}{t^{\frac{1}{2}}}\norm{c_0}_{L^{\infty}}
+\int_0^t\norm{\nabla e^{(t-s)\Delta} k n}_{L^{\infty}}(s)ds
\]
\[
+\int_0^t\norm{\nabla e^{(t-s)\Delta}(u\nabla c)}_{L^{\infty}}(s)ds
=\frac{1}{t^{\frac{1}{2}}}\norm{c_0}_{L^{\infty}}+J_1+J_2.
\]
Firstly, we estimate $J_1$.
\[
J_1\lesssim \int_0^{t/2}\frac{1}{(t-s)^{\frac{3}{2}}}\norm{k
n(s)}_{L^1}ds+\int_{t/2}^t\frac{1}{(t-s)^{\frac{1}{2}}}\norm{k
n(s)}_{L^{\infty}}ds
\]
\begin{equation}\label{Dec03-CKL130}
\lesssim
\frac{1}{t^{\frac{1}{2}}}\norm{k(c)}_{L^{\infty}}\norm{n}_{L^1}
+\frac{1}{t^{\frac{1}{2}}}\norm{k(c)}_{L^{\infty}}\norm{n}_{\calK_{\infty}(\R^2)}\lesssim
\frac{\epsilon_1}{t^{\frac{1}{2}}}
+\frac{\epsilon_1}{t^{\frac{1}{2}}}\norm{n}_{\calK_{\infty}(\R^2)}.
\end{equation}
Before we estimate $J_2$, we set $1/4^+=1/4-1/\beta$ and
$1/4^-=1/4+1/\beta$ with $\beta>4$. We then estimate $J_2$.
\[
J_2\lesssim \int_0^{t/2}\frac{1}{t-s}\norm{u\nabla
c}_{L^2}ds+\int_{t/2}^t\frac{1}{(t-s)^{\frac{3}{2}-\frac{1}{2^{-}}}}
\norm{u\nabla c}_{L^{2^{+}}}(s)ds
\]
\[
\lesssim \frac{1}{t}\int_0^{t/2}\norm{u}_{L^{4^+}}\norm{\nabla
c}_{L^{4^-}}ds+\int_{t/2}^t\frac{1}{(t-s)^{\frac{3}{2}-\frac{1}{2^{-}}}}
\norm{u}_{L^{2^{+}}}\norm{\nabla c}_{L^{\infty}}(s)ds
\]
\[
\lesssim
\frac{1}{t}\int_0^{t/2}\norm{\omega}_{L^{\frac{4\beta}{3\beta-4}}}\norm{\nabla
c}_{L^{4^-}}ds+\int_{t/2}^t\frac{1}{(t-s)^{\frac{3}{2}-\frac{1}{2^{-}}}}
\norm{\omega}_{L^{\frac{\alpha}{\alpha-1}}}\norm{\nabla
c}_{L^{\infty}}(s)ds
\]
\begin{equation}\label{Dec31-CKL30}
\lesssim
\frac{1}{t^{\frac{1}{2}}}\norm{\omega}_{\calK_{\frac{4\beta}{3\beta-4}}(\R^2)}
\norm{ c}_{\calN_{4^-}(\R^2)}+\frac{1}{t^{\frac{1}{2}}}
\norm{\omega}_{\calK_{\frac{\alpha}{\alpha-1}}(\R^2)} \norm{
c}_{\calN_{\infty}(\R^2)}\lesssim\frac{\epsilon^2_1}{t^{\frac{1}{2}}}
+\frac{\epsilon_1}{t^{\frac{1}{2}}}\norm{
c}_{\calN_{\infty}(\R^2)},
\end{equation}
where the result in Lemma \ref{lemma-ncw-100} is used. Combining
\eqref{Dec03-CKL130} and \eqref{Dec31-CKL30}, we have
\begin{equation}\label{Dec31-CKL40}
\norm{\nabla c}_{L^{\infty}}(t)\lesssim \frac{\epsilon_1}{t^{\frac{1}{2}}}
+\frac{\epsilon_1}{t^{\frac{1}{2}}}\norm{n}_{\calK_{\infty}(\R^2)}
+\frac{\epsilon_1}{t^{\frac{1}{2}}}\norm{
c}_{\calN_{\infty}(\R^2)}.
\end{equation}
Next, we estimate the vorticity. For any $1\le r< \infty$
\[
\norm{\omega (t)}_{L^r} \lesssim t^{-1+\frac{1}{r}}
\norm{\omega_0}_{L^1} +\int_0^t
\norm{\nabla^{\perp}e^{(t-s)\tilde{\Delta}}\bke{n(s)\nabla\phi}}_{L^r}ds
\]
\[
+\int_0^t \norm{\nabla
e^{(t-s)\tilde{\Delta}}\bke{u\omega}}_{L^r}ds= t^{-1+\frac{1}{r}}
\norm{\omega_0}_{L^1} +L_1+L_2.
\]
If we restrict $ 2\le r$,  we have
\[
L_1\lesssim \int_0^{t/2}
\frac{1}{(t-s)^{\frac{3}{2}-\frac{1}{r}}}\norm{n(s)}_{L^2}
\norm{\nabla \phi }_{L^{2}} +\int_{t/2}^t
\frac{1}{(t-s)^{1-\frac 1r}}\| n(s)\|_{L^{\infty}}\| \na \phi \|_{L^2}\]
\[
\lesssim \int_0^{t/2}
\frac{1}{(t-s)^{\frac{3}{2}-\frac{1}{r}}}\norm{n(s)}^{\frac{1}{2}}_{L^1}
\norm{n(s)}^{\frac{1}{2}}_{L^\infty} \norm{\nabla \phi }_{L^{2}}
\]
\begin{equation}\label{Dec03-CKL140}
+\int_{t/2}^t \frac{1}{(t-s)^{1-\frac 1r}}\norm{n(s)}_{L^{\infty}}
\norm{\nabla \phi }_{L^{2}} \lesssim
\frac{\epsilon_1}{t^{1-\frac{1}{r}}}
+\frac{1}{t^{1-\frac{1}{r}}}\norm{n}_{\calK_{\infty}(\R^2)},
\end{equation}
where we used the
H\"older's inequality and Young's inequality. The exponents $r^*, \tilde r$ are defined by
$1/r^*=1/2-1/r$ and
$1/r^*=1/\tilde{r}-1/2$. Now we estimate $L_2$.
\[
L_2\lesssim\int_0^{t/2}
\frac{1}{(t-s)^{\frac{3}{2}-\frac{1}{2^-}}}\norm{u}_{L^{2^+}}
\norm{\omega}_{L^{r}} +\int_{t/2}^t
\frac{1}{(t-s)^{1-\frac 1r}}\norm{u}_{L^{r^*}}
\norm{\omega}_{L^{r}}
\]
\[
\lesssim\frac{1}{t^{\frac{3}{2}-\frac{1}{2^-}}}\int_0^{t/2}
\norm{\omega}_{L^{\frac{\alpha}{\alpha-1}}} \norm{\omega}_{L^{r}}
+\int_{t/2}^t
\frac{1}{(t-s)^{1-\frac 1r}}\norm{\omega}_{L^{\tilde{r}}}
\norm{\omega}_{L^{r}}
\]
\begin{equation}\label{Dec31-CKL50}
\lesssim\frac{1}{t^{1-\frac{1}{r}}}\norm{\omega}_{\calK_{\frac{\alpha}{\alpha-1}}(\R^2)}
\norm{\omega}_{\calK_r(\R^2)}+
\frac{1}{t^{1-\frac{1}{r}}}\norm{\omega}_{\calK_{\tilde{r}}(\R^2)}
\norm{\omega}_{\calK_r(\R^2)}\lesssim\frac{\epsilon_1}{t^{1-\frac{1}{r}}}\norm{\omega}_{\calK_r(\R^2)},
\end{equation}
where the result in Lemma \ref{lemma-ncw-100} is used. Therefore, we
have
\begin{equation}\label{Dec31-CKL60}
\norm{\omega (t)}_{L^r} \lesssim\frac{\epsilon_1}{t^{1-\frac{1}{r}}}
+\frac{1}{t^{1-\frac{1}{r}}}\norm{n}_{\calK_{\infty}(\R^2)}
+\frac{\epsilon_1}{t^{1-\frac{1}{r}}}\norm{\omega}_{\calK_r(\R^2)}.
\end{equation}
Using the estimate \eqref{Dec31-CKL20}, we obtain
\begin{equation}\label{Jan12-CKL100}
\norm{\omega}_{\calK_r(\R^2)}\lesssim
\epsilon_1+\epsilon_1\norm{n}_{\calK_{\infty}(\R^2)}+\epsilon_1\norm{
c}_{\calN_{\infty}}+\norm{n}_{\calK_{\infty}}\norm{
c}_{\calN_{\infty}}+\epsilon_1\norm{\omega}_{\calK_r(\R^2)}.
\end{equation}
Combining estimates \eqref{Dec31-CKL20}, \eqref{Dec31-CKL40} and
\eqref{Jan12-CKL100}, we obtain
\begin{equation}\label{Dec03-CKL150}
\norm{n}_{\calK_{\infty}}+\norm{
c}_{\calN_{\infty}}+\norm{\omega}_{\calK_r}\lesssim \epsilon_1+
\epsilon_1(\norm{n}_{\calK_{\infty}}+\norm{
c}_{\calN_{\infty}}+\norm{\omega}_{\calK_r})+\norm{n}_{\calK_{\infty}}\norm{
c}_{\calN_{\infty}}.
\end{equation}
This completes the proof.
\end{proof}
We remark that the case $r=\infty$ in \eqref{Dec3-CKL20} is missing
due to Sobolev embedding inequalities.

%

Next we show estimates of higher derivatives. For convenience, we
denote
\[
\norm{\nabla n}_{\calK^1_{\infty}(\R^2)}:=\sup_{t\ge
0}t^{\frac{3}{2}}\norm{\nabla n(t)}_{L^{\infty}(\R^2)},
\qquad\norm{\nabla^2 c}_{\calK_{\infty}(\R^2)}:= \sup_{t\ge
0}t\norm{\nabla^2 c(t)}_{L^{\infty}(\R^2)},
\]
\[
\norm{\nabla\omega}_{\calK^1_r(\R^2)}:=\sup_{t\ge 0}t^{\frac
32-\frac{1}{r}}\norm{\nabla\omega(t)}_{L^r(\R^2)},\qquad 1\le r<2.
\]
\begin{lemma}\label{CKL-Dec05-100}
Let $n, c$ and $\omega$ be solutions of
\eqref{CKL20-dec12}-\eqref{CKL20-jan9}. If the assumptions in
Theorem \ref{Theorem1} are satisfied, then
\begin{equation}\label{CKL-Dec05-200}
\norm{\nabla^2 c}_{L^{\infty}}(t)\le \frac{C\epsilon_1}{t},\qquad
\norm{\nabla n}_{L^{\infty}}(t)\le
\frac{C\epsilon_1}{t^{\frac{3}{2}}},
\end{equation}
\begin{equation}\label{CKL30-jan9}
\norm{\nabla \omega}_{L^{r}}(t)\le
\frac{C\epsilon_1}{t^{\frac{3}{2}-\frac{1}{r}}},\qquad 1\le r<2.
\end{equation}
\end{lemma}
\begin{proof}
We first estimate $\nabla^2 c$.
\[
\norm{\nabla^2 c}(t)\lesssim \frac{1}{t}\norm{c_0}_{L^{\infty}}
+\int_0^{\frac{t}{2}}\norm{\nabla^2 e^{(t-s)\Delta} k
n}_{L^{\infty}}(s)ds+\int_{\frac{t}{2}}^t\norm{\nabla
e^{(t-s)\Delta} \nabla(k n)}_{L^{\infty}}(s)ds
\]
\[
+\int_0^{\frac{t}{2}}\norm{\nabla^2 e^{(t-s)\Delta} u\nabla
c}_{L^{\infty}}(s)ds+\int_{\frac{t}{2}}^t\norm{\nabla
e^{(t-s)\Delta} \nabla(u\nabla c)}_{L^{\infty}}(s)ds.
\]
Consider the second term in the rightside.
\[
\int_0^{\frac{t}{2}}\norm{\nabla^2 e^{(t-s)\Delta} k
n}_{L^{\infty}}(s)ds\lesssim
\int_0^{\frac{t}{2}}\frac{1}{(t-s)^2}\norm{
n}_{L^{1}}\norm{k(c)}_{L^{\infty}}ds\lesssim \frac{\epsilon_1}{t}.
\]
The third term is estimated as follows:
\[
\int_{\frac{t}{2}}^t\norm{\nabla e^{(t-s)\Delta} \nabla(k
n)}_{L^{\infty}}(s)ds\lesssim\int_{\frac{t}{2}}^t\frac{1}{(t-s)^{\frac{1}{2}}}[\norm{k'(c)}_{L^{\infty}}
\norm{\nabla c n}_{L^{\infty}} +\norm{k(c)}_{L^{\infty}}\norm{\nabla
n}_{L^{\infty}}](s)ds
\]
\[
\lesssim
\norm{k'(c)}_{L^{\infty}}\int_{\frac{t}{2}}^t\frac{1}{(t-s)^{\frac{1}{2}}s^{\frac{3}{2}}}ds
+\norm{k(c)}_{L^{\infty}}\norm{\nabla
n}_{\calK^1_{\infty}}\int_{\frac{t}{2}}^t\frac{1}{(t-s)^{\frac{1}{2}}s^{\frac{3}{2}}}ds
\lesssim \frac{\epsilon_1}{t}+\frac{\epsilon_1}{t}\norm{\nabla
n}_{\calK^1_{\infty}}.
\]
We estimate the fourth and fifth terms.
\[
\int_0^{\frac{t}{2}}\norm{\nabla^2 e^{(t-s)\Delta} u\nabla
c}_{L^{\infty}}(s)ds\lesssim
\int_0^{\frac{t}{2}}\frac{1}{(t-s)^{\frac{5}{3}}} \norm{u\nabla
c}_{L^{\frac{3}{2}}}(s)ds
\]
\[
\lesssim \int_0^{\frac{t}{2}}\frac{1}{(t-s)^{\frac{5}{3}}}
\norm{u}_{L^3}\norm{\nabla c}_{L^{3}}(s)ds\lesssim
\int_0^{\frac{t}{2}}\frac{1}{(t-s)^{\frac{5}{3}}}
\norm{\omega}_{L^{\frac{6}{5}}}\norm{\nabla c}_{L^{3}}(s)ds\leq
\frac{\epsilon_1}{t}.
\]
For $p>2$ and $1<q<2$ with $1/p+1/q=1$
\[
\int_{\frac{t}{2}}^t\norm{\nabla e^{(t-s)\Delta} \nabla(u\nabla
c)}_{L^{\infty}}(s)ds\lesssim \int_{\frac{t}{2}}^t
\frac{1}{(t-s)^{\frac{3}{2}-\frac{1}{q}}}(\norm{\nabla u\nabla
c}_{L^p}+\norm{u\nabla^2 c}_{L^p})(s)ds
\]
\[
\lesssim \int_{\frac{t}{2}}^t
\frac{1}{(t-s)^{\frac{3}{2}-\frac{1}{q}}}(\norm{\omega}_{L^p}\norm{\nabla
c}_{L^{\infty}}+\norm{u}_{L^p}\norm{\nabla^2 c}_{L^{\infty}})(s)ds
\lesssim \frac{\epsilon_1}{t}+\frac{\epsilon_1}{t}\norm{\nabla^2
c}_{\calK_{\infty}}.
\]
Summing up all estimates, we obtain
\begin{equation}\label{CKL30-june15}
\norm{\nabla^2 c}_{\calK_{\infty}}\lesssim
\epsilon_1+\epsilon_1\norm{\nabla^2
c}_{\calK_{\infty}}+\epsilon_1\norm{\nabla n}_{\calK^1_{\infty}}.
\end{equation}
\indent
Next we consider $\nabla n$.
\[
\norm{\nabla n}_{L^{\infty}}(t)\lesssim
\frac{1}{t^{\frac{3}{2}}}\norm{n_0}_{L^1}
+\int_1^{\frac{t}{2}}\norm{\nabla^2 e^{(t-s)\Delta}[\chi n\nabla
c]}_{L^{\infty}}(s)ds+\int_{\frac{t}{2}}^t\norm{\nabla
e^{(t-s)\Delta}\nabla[\chi n\nabla c]}_{L^{\infty}}(s)ds
\]
\[
+\int_0^{\frac{t}{2}}\norm{\nabla^2 e^{(t-s)\Delta}[u
n]}_{L^{\infty}}(s)ds+\int_{\frac{t}{2}}^t\norm{\nabla
e^{(t-s)\Delta}[u\nabla n]}_{L^{\infty}}(s)ds.
\]
First, we compute
\[
\int_1^{\frac{t}{2}}\norm{\nabla^2 e^{(t-s)\Delta}[\chi n\nabla
c]}_{L^{\infty}}(s)ds\leq
\int_1^{\frac{t}{2}}\frac{1}{(t-s)^2}\norm{n\nabla c}_{L^1} \lesssim
\frac{\epsilon_1}{t^2}\int_1^{\frac{t}{2}}\frac{1}{s^{\frac{1}{2}}}ds
\lesssim\frac{\epsilon_1}{t^{\frac{3}{2}}}.
\]
Secondly,
\[
\int_{\frac{t}{2}}^t\norm{\nabla e^{(t-s)\Delta}\nabla[\chi n\nabla
c]}_{L^{\infty}}(s)ds\leq
\int_{\frac{t}{2}}^t\frac{1}{(t-s)^{\frac{1}{2}}}(\norm{\nabla
n\nabla c}_{L^{\infty}}+\norm{n\nabla^2
c}_{L^{\infty}}+\norm{n\abs{\nabla c}^2}_{L^{\infty}})(s)ds
\]
\[
\lesssim \norm{\nabla
n}_{\calK^1_{\infty}}\int_{\frac{t}{2}}^t\frac{\epsilon_1}{(t-s)^{\frac{1}{2}}s^{2}}ds
+\norm{\nabla^2
c}_{\calK_{\infty}}\int_{\frac{t}{2}}^t\frac{\epsilon_1}{(t-s)^{\frac{1}{2}}s^2}ds
+\int_{\frac{t}{2}}^t\frac{\epsilon_1}{(t-s)^{\frac{1}{2}}s^2}ds
\]
\[
\lesssim \frac{\epsilon_1}{t^{\frac{3}{2}}}\norm{\nabla
n}_{\calK^1_{\infty}}+\frac{\epsilon_1}{t^{\frac{3}{2}}}\norm{\nabla^2
c}_{\calK_{\infty}} +\frac{\epsilon_1}{t^{\frac{3}{2}}}.
\]
Thirdly,
\[
\int_0^{\frac{t}{2}}\norm{\nabla^2 e^{(t-s)\Delta}[u
n]}_{L^{\infty}}(s)ds\lesssim
\int_0^{\frac{t}{2}}\frac{1}{(t-s)^{\frac{5}{3}}}
\norm{un}_{L^{\frac{3}{2}}}(s)ds
\]
\[
\lesssim \int_0^{\frac{t}{2}}\frac{1}{(t-s)^{\frac{5}{3}}}
\norm{u}_{L^3}\norm{n}_{L^{3}}(s)ds\lesssim
\int_0^{\frac{t}{2}}\frac{1}{(t-s)^{\frac{5}{3}}}
\norm{\omega}_{L^{\frac{6}{5}}}\norm{n}_{L^{3}}(s)ds\leq
\frac{\epsilon_1}{t^{\frac{3}{2}}}.
\]
Lastly, for $p>2$ and $1<q<2$ with $1/p+1/q=1$
\[
\int_{\frac{t}{2}}^t\norm{\nabla e^{(t-s)\Delta}[u\nabla
n]}_{L^{\infty}}(s)ds\lesssim \int_{\frac{t}{2}}^t
\frac{1}{(t-s)^{\frac{3}{2}-\frac{1}{q}}}\norm{ u\nabla
n}_{L^p}(s)ds
\]
\[
\lesssim \int_{\frac{t}{2}}^t
\frac{1}{(t-s)^{\frac{3}{2}-\frac{1}{q}}}\norm{u}_{L^p}(s)\norm{\nabla
n}_{L^{\infty}}(s)ds\lesssim\frac{\epsilon_1}{t^{\frac{3}{2}}}\norm{\nabla
n}_{\calK^1_{\infty}}.
\]
Summing up, we obtain
\begin{equation}\label{CKL40-june15}
\norm{\nabla n}_{\calK^1_{\infty}}\lesssim \epsilon_1\norm{\nabla
n}_{\calK^1_{\infty}}+\epsilon_1\norm{\nabla^2 c}_{\calK_{\infty}}
+\epsilon_1.
\end{equation}
Combining \eqref{CKL30-june15} and \eqref{CKL40-june15}, we obtain the first assertion of the lemma:
\bq\label{lem5-1}
\| \na n\|_{\calK^1_{\infty}} + \| \na^2 c\|_{\calK_{\infty}} \le C \ep_1.
\eq
 With the above estimate in hands,  it is easy to show $\| \na n\|_{L^2}$ satisfy the following decay:
\begin{equation} \label{lem5-m}
\| \na n\|_{L^2}(t) \le \frac{\ep_1}{t}.
\end{equation}
\indent
We consider the vorticity equation. Using the integral
representation, we compute
\[
\norm{\nabla \omega}_{L^{r}(t)}\lesssim
\frac{1}{t^{\frac{3}{2}-\frac{1}{r}}}\norm{\omega_0}_{L^1}
+\int_0^{\frac{t}{2}}\norm{\nabla^2 e^{(t-s)\Delta}[u
\omega]}_{L^{r}}(s)ds+\int_{\frac{t}{2}}^t\norm{\nabla
e^{(t-s)\Delta}[u\nabla \omega]}_{L^{r}}(s)ds
\]
\[
+\int_0^{\frac{t}{2}}\norm{\nabla^2 e^{(t-s)\Delta}[n\nabla
\phi]}_{L^{r}}(s)ds+\int_{\frac{t}{2}}^t\norm{\nabla
e^{(t-s)\Delta}\nabla[ n\nabla \phi]}_{L^{r}}(s)ds.
\]
First, for $p>2$ and $1<q<2$ with $1/p+1/q=1$
\[
\int_0^{\frac{t}{2}}\norm{\nabla^2 e^{(t-s)\Delta}[u
\omega]}_{L^{r}}(s)ds\lesssim
\int_0^{\frac{t}{2}}\frac{1}{(t-s)^{2-\frac{1}{r}}}\norm{u}_{L^p}\norm{
\omega}_{L^{q}}(s)ds\lesssim
\frac{\epsilon_1}{t^{\frac{3}{2}-\frac{1}{r}}}.
\]
Secondly,
\[
\int_{\frac{t}{2}}^t\norm{\nabla e^{(t-s)\Delta}[u\nabla
\omega]}_{L^{r}}(s)ds\lesssim
\int_{\frac{t}{2}}^t\frac{1}{(t-s)^{\frac{3}{2}-\frac{1}{q}}}\norm{u}_{L^p}\norm{
\nabla\omega}_{L^{r}}(s)ds\lesssim
\frac{\epsilon_1}{t^{\frac{3}{2}-\frac{1}{r}}}\norm{\nabla\omega}_{\calK^1_r}.
\]
Thirdly,
\[
\int_0^{\frac{t}{2}}\norm{\nabla^2 e^{(t-s)\Delta}[n\nabla
\phi]}_{L^{r}}(s)ds\lesssim
\int_0^{\frac{t}{2}}\frac{1}{(t-s)^{2-\frac{1}{r}}}\norm{n}_{L^2}\norm{
\nabla\phi}_{L^{2}}(s)ds\leq
\frac{\epsilon_1}{t^{\frac{3}{2}-\frac{1}{r}}}.
\]
Lastly,
\[
\int_{\frac{t}{2}}^t\norm{\nabla e^{(t-s)\Delta}\nabla[ n\nabla
\phi]}_{L^{r}}(s)ds\lesssim \int_{\frac{t}{2}}^t\norm{\nabla
e^{(t-s)\Delta}[ \nabla n\nabla \phi+ n\nabla^2 \phi]}_{L^{r}}(s)ds
\]
\[
\lesssim \int_{\frac{t}{2}}^t\norm{\nabla e^{(t-s)\Delta}[ \nabla
n\nabla \phi+ n\nabla^2 \phi]}_{L^{r}}(s)ds \lesssim
\int_{\frac{t}{2}}^t\frac{1}{(t-s)^{\frac{3}{2}-\frac 1r}}\norm{\nabla
n}_{L^{2}}\norm{\nabla \phi}_{L^{2}}ds
\]
\[
+\int_{\frac{t}{2}}^t\frac{1}{(t-s)^{\frac{3}{2}-\frac{1}{r}}}
\norm{n}_{L^\infty}\norm{\nabla^2 \phi}_{L^{1}}ds\lesssim
\frac{\epsilon_1}{t^{\frac 32-\frac{1}{r}}}
\]
by \eqref{lem5-m} and Lemma \ref{lemma-higher-est}.
Summing up, we obtain
\begin{equation}\label{CKL10-Dec05}
\norm{\nabla\omega}_{\calK^1_r}\lesssim \epsilon_1+\epsilon_1\norm{\nabla\omega}_{\calK^1_r}.
\end{equation}
This completes the proof of the second assertion of Lemma \ref{CKL-Dec05-100}.
\end{proof}

\begin{remark}
The restriction that $r<2$ in \eqref{CKL30-jan9} is due to absence
of temporal decay of $\phi$, since $\phi$ is independent of time. We
leave it open question whether or not the estimate
\eqref{CKL30-jan9} is available for $r\ge 2$.
\end{remark}

\begin{pfprop1}
The decay estimate of solutions is the consequence of consecutive
Lemma \ref{lemma-ncw-100}-Lemma \ref{CKL-Dec05-100}.
\end{pfprop1}

\section{Proof of Theorem \ref{Theorem1}}

In this section, we present the proof of Theorem \ref{Theorem1}.

\begin{pfthm1}
We define the family of rescaled solutions in $\R^2$ \footnote{ {
$(n_k, c_k, u_k)$ solve system \eqref{KSNS} with the potential
$\phi_k$, instead of $\phi$.
}}
\[
n_{k}(x,t)=k^2 n(kx, k^2 t),\quad c_k(x,t)=c(kx, k^2 t),\quad
u_k(x,t)=k u(kx, k^2 t),\qquad \phi_k(x)=\phi(kx)
\]
with (sufficiently regular) initial data
\[
n_{k,0}(x)=k^2 n_0(kx),\qquad c_{k,0}(x)=c_0(kx),\qquad u_{k,0}(x)=k
u_0(kx).
\]
For the vorticity field, we have following rescaled solutions and
initial data
\[
\omega_{k}(x,t)=k^2 \omega(kx, k^2 t),\qquad \omega_{k,0}(x)=k^2
\omega_0(kx).
\]
We recall some invariant quantities (independent of $k$),which are
\[
\norm{n_k(t)}_{L^1}=\norm{n(t)}_{L^1}=\norm{n_0}_{L^1},\qquad \|
c_{k, 0} \| _{L^{\infty}} = \|c_0\|_{L^{\infty}},
\]
\[
\norm{\omega_{k,0}}_{L^1}=\norm{\omega_0}_{L^1},\qquad \int_{\R^2}
\omega(t)dx=\int_{\R^2} \omega_k(t)dx=\int_{\R^2} \omega_0dx.
\]
Therefore, the smallness assumption \eqref{CKL-assumption10}
is likewise valid for $(n_{k,0}, c_{k, 0}, \om_{k,0})$, namely
\[
\norm{n_{k,0}}_{L^1(\R^2)}+\norm{c_{k,0}}_{L^{\infty}(\R^2)}
+\norm{\omega_{k,0}}_{L^1(\R^2)}
<\epsilon_1.
\]
We also note that the potential ${\phi_k}$ also remains invariant by
norm of
\begin{equation}\label{CKL-Feb19-10}
\| \na \phi_k \|_{L^2} = \| \na \phi \|_{L^2}.
\end{equation}
From now on, we consider the vorticity equation, instead equation of
velocity fields. We then have global existence and time decay of
solutions $(n_k, c_k, \omega_k)$ and sequence of functions also
solves the system in a weak sense as follows: (possibly subsequence)
for $\varphi \in C_0^{\infty}(\bbr^2 \times [0, \infty))$ it holds
\begin{align} \label{wn_k} \begin{aligned}
\int_0^{\infty}\int_{\R^2} (\pa_t \varphi  +\Del \varphi) n_k + n_k
u_k\cdot\nabla \varphi + \chi(c_k)
n_k\nabla c_k\nabla \varphi \,dxdt  =\int_{\R^2}n_{k,0}\varphi(x,0)dx,\\
\int_0^{\infty} \int_{\bbr^2} (\pa_t \varphi + \Del \varphi) c_k +
c_k u_k\cdot\nabla \varphi - k(c_k) n_k\varphi dxdt =
 \int_{\R^2}c_{k,0}\varphi(x,0)dx,
\\
\int_0^{\infty} \int_{\R^2} (\pa_t \varphi + \Del \varphi) \om_k +
\om_k u_k\cdot\nabla \varphi + n_k\nabla\phi_k\nabla^{\perp}\varphi
dxdt = \int_{\R^2}\omega_k(x ,0) \varphi(x,0 )dx.
\end{aligned}\end{align}
In particular the time decay rates in Proposition
\ref{time-decay}  are scaling invariant, so
 rescaled solutions also satisfy uniform estimates
\begin{align} \label{CKL100-jan14}
\| n_k(t)\|_{L^{\infty}(\R^2)} \le \frac{C\epsilon_1}{t},&\qquad \|
\nabla n_k(t)\|_{L^{\infty}(\R^2)} \le
\frac{C\epsilon_1}{t^{\frac{3}{2}}},
\\\label{CKL110-jan14}
\|\nabla c_k(t)\|_{L^{\infty}(\R^2)} \le
\frac{C\epsilon_1}{t^{\frac{1}{2}}},&\qquad \|\nabla^2
c_k(t)\|_{L^{\infty}(\R^2)} \le \frac{C\epsilon_1}{t},
\\\label{CKL120-jan14}
\| \omega_k(t)\|_{L^{r}(\R^2)} \le
\frac{C\epsilon_1}{t^{1-\frac{1}{r}}} \quad 1<r<\infty, &\qquad \|
\na \omega_k(t)\|_{L^r(\R^2)}\le \frac{C\epsilon_1}{t^{\frac 32
-\frac{1}{r}}}\quad1\le r< 2.
\end{align}
Therefore, we have strong convergence of $(n_k, c_k, \omega_k)$ in
$L^p\times W^{1,p}\times L^r$ with $1\le p<\infty$ and  $ 1\le r
<\infty $ in any compact set in $\R^2\times (0,\infty)$. Let us
denote limit functions by $(\tilde{n}, \tilde{c}, \tilde{\omega})$
as $k\rightarrow\infty$ (possibly subsequence of $k$).  To be more
precise, there is a subsequence such that as $k_j$ tends to
infinity, for any $1\leq p< \infty$ ,  $ 1\le r <\infty$ and for all
$R, \eta_{\ep}>0$
\[
n_{k_j}\quad\longrightarrow\quad \tilde{n} \quad\mbox{ strongly in
}\,\,L^p(B_R\times (\eta_{\ep}, \eta_{\ep}^{-1})),
\]
\[
\nabla c_{k_j}\quad\longrightarrow\quad \nabla\tilde{c} \quad\mbox{
strongly in }\,\,L^p(B_R\times (\eta_{\ep}, \eta_{\ep}^{-1})),
\]
\[
\omega_{k_j}\quad\longrightarrow\quad \tilde{\omega} \quad\mbox{
strongly in }\,\,L^r(B_R\times (\eta_{\ep}, \eta_{\ep}^{-1})).
\]
We  observe that $(\tilde{n}, \tilde{c}, \tilde{\omega})$ satisfy
the estimates \eqref{CKL100-jan14}-\eqref{CKL120-jan14} . Similarly
we denote by $\tilde \phi$   the weak limit of $ \phi_k$, then
$\na \tilde \phi \in L^2(\bbr)$
 due to
\eqref{CKL-Feb19-10}.  Combining
the strong convergence in any compact domain of $\bbr^2 \times (0,
\infty)$ with these time decays, we can take the limit $k \to
\infty$ to \eqref{wn_k}, and show that $(\tilde{n},
\tilde{c}, \tilde{\omega})$ solve the following equations  in a weak
sense:
\begin{align}\label{tilde}\begin{cases}
\partial_t \tilde{n} + \tilde{u} \cdot \nabla  \tilde{n} - \Delta \tilde{n}
= -\nabla\cdot (\chi (\tilde{c}) \tilde{n} \nabla \tilde{c}),\\
\partial_t \tilde{c} + \tilde{u} \cdot \nabla \tilde{c}-\Delta \tilde{c} =-k(\tilde{c}) \tilde{n},\\
\partial_t \tilde{\omega} + \tilde{u}\cdot \nabla \tilde{\omega}
-\Delta \tilde{\omega} =-\nabla\times(\tilde{n} \nabla \tilde{\phi})
\end{cases}
\end{align}
with initial data
\begin{equation}\label{CKL230-jan14}
\tilde{n}_0=m\delta_0,\quad \tilde{c}_0=0,\quad
\tilde{\omega}_0=\gamma\delta_0,
\end{equation}
where $m$ is the total mass of $n$ and $\gamma$ is total circulation
of $\omega$. While the proof for passing to limit goes on closely
following \cite[section 2.5.1]{GGS},
 for the sake of concreteness we take some terms, say,
  $\int_0^{\infty}\int _{\bbr^2}\chi(c_k) n_k\nabla c_k\nabla \varphi \,dxdt $ and
  $\int_0^{\infty} \int_{\bbr^2} \om_k u_k \na \varphi dxdt$  to show
  \[ \lim_{k\to \infty} \int_0^{\infty} \int_{\bbr^2}\chi(c_k)
n_k\nabla c_k\nabla \varphi \,dxdt
  = \int_0^{\infty} \int_{\bbr^2}\chi(\tilde c)
\tilde n \nabla \tilde c \nabla \varphi \,dxdt , \]
\[  \lim_{k\to \infty} \int_0^{\infty} \int_{\bbr^2} \om_k u_k \na \varphi dxdt =
   \int_0^{\infty} \int_{\bbr^2} \tilde \om \tilde u \na \varphi  dxdt.
\]
  Let  ${\rm supp}\, \varphi  \in B_R\times [0, T]$. We define
\[
  F_k(t) = \int_{B_R} \chi(c_k) n_k \na c_k \na \varphi dx, \qquad
  F(t)=\int_{B_R}\chi(\tilde c)
\tilde n \nabla \tilde c \nabla \varphi.
\]
Due to strong convergence we have $F_k(t) \to F(t) $ for $ t>0$.
Using the decay estimate \eqref{CKL110-jan14}, it holds that $F_k(t)
\le C(R) t^{-\frac 12}$, and we then have $\displaystyle \lim_{k\to
0} F_k(t) = F(t)$ via the dominated convergence theorem. For the
second example we also have
\[
\int_{B(R)} \om_k u_k \na \varphi dx \le
\| \om_k\|_{L^{\frac{4}{3}}}\|u_k\|_{L^4} \| \na
\varphi\|_{L^{\infty}} \le C(R) t^{-\frac 12},
\]
where we used the embedding $\|u_k\|_{L^4}\le C\| \om_k\|_{L^{\frac
43}}$ and the estimate \eqref{CKL120-jan14}. In fact, it holds that
\begin{equation}\label{CKL-Feb19-20}
\tilde c =0, \qquad \na \tilde \phi = 0.
\end{equation}
Indeed, from the $c_k$
equation we have
\[ \|c_k(t)\|_{L^p}  \le \|c_{k, 0}\|_{L^p} = k^{-\frac 2p}\|c_0\|_{L^p},\qquad 1\le p \le \infty.\]
It implies $\tilde c = 0$. Next we show that $\tilde \phi$ is a function of homogeneity zero.
 If $l >0$ is fixed and  $\psi \in C_0^{\infty}(\bbr^2)$, we have
 \[ \lim_{k\to \infty}\int _{\bbr^2}\phi_k(lx) \psi(x) dx = \lim_{k\to \infty} \int _{\bbr^2}\phi(kl x) \psi(x) dx
 = \int _{\bbr^2} \tilde \phi (x) \psi(x) dx.\]
On the other hand, denoting $\psi_l(y) := \psi(l^{-1}y)$, we see
that
 \begin{align*}
 \lim_{k\to \infty}\int _{\bbr^2}\phi_k(lx) \psi(x) dx & =
 \lim_{k\to \infty} \int_{\bbr^2} l^{-2}\phi_k(y) \psi(l^{-1} y) dy \\
 & = l^{-2} \int_{\bbr^2} \tilde\phi (y) \psi_l(y) dy = \int _{\bbr^2} \tilde \phi(lx) \psi(x) dx.
 \end{align*}
Therefore, $\na \tilde  \phi $ is a function of homogeneity  1,
namely $\nabla \phi(x)=l \nabla\phi (l x)$, which implies $\na
\tilde\phi =0$, since $\na \tilde \phi \in L^2(\bbr^2)$. On account
of \eqref{CKL-Feb19-20}, the system \eqref{tilde}-
\eqref{CKL230-jan14} is reduced to
\begin{align*}
\begin{cases}\partial_t \tilde{n} + \tilde{u} \cdot \nabla  \tilde{n} - \Delta \tilde{n}
= 0,\\
\partial_t \tilde{\omega} + \tilde{u}\cdot \nabla \tilde{\omega}
-\Delta \tilde{\omega} =0
\end{cases}\end{align*}
with initial data
\[ \tilde n_0 = m\del_0, \quad \tilde \om_0 = \ga\del_0.\]
  It is well established that
the vorticity equation of Navier-Stokes equation with the dirac-delta initial data
has the unique solution
\[
\tilde{w}(x,t)=\gamma
\Gamma(x,t).
\] We refer to \cite{GW} and \cite{GGS}, and references cited therein.
In
particular
\begin{equation*}
\tilde{u}=K*\tilde{\omega},\qquad K(x)=\nabla^{\perp} \log
|x|=<-\frac{x_2}{\abs{x}^2}, \frac{x_1}{\abs{x}^2}>,
\end{equation*}
which implies
$ \tilde u \cdot \na \tilde n=0$
by Lemma \ref{CKL-GGS-100} .  Then $\tilde n$ equation is reduced to
\[
\partial_t \tilde{n}  - \Delta \tilde{n}
=  0\]
with initial data
$ \tilde{n}_0=m\delta_0$.
As a direct application of Theorem 4.4.2 in \cite{GGS}, the above equation has
the unique solution
\[  \tilde n (x,t)= m \Gamma (x,t).\]

The asymptotics are obtained as follows. When $t=1$,
tending to zero as $k_j\rightarrow\infty$, we have
\begin{equation}\label{CKL300-Dec4}
\lim_{k_j\rightarrow
\infty}\norm{n_{k_j}(\cdot,1)-\tilde{n}(\cdot,1)}_{L^{\infty}(B_R)}=0.
\end{equation}
Using $\tilde{n}=m\Gamma$ is self-similar, we observe that
\[
n_{k_j}(x,1)-\tilde{n}(x,1)=k^2_j n(k_j x, k^2_j)-k^2_j\tilde{n}(k_j
x, k^2_j).
\]
Setting $t=k^2_j$, \eqref{CKL300-Dec4} can be rewritten as
\begin{equation}\label{CKL310-Dec4}
t\norm{\bke{n(\cdot,t)-\tilde{n}(\cdot,t)}}_{L^{\infty}(B_{t,R})}\,\,\longrightarrow\,\,
0\qquad \mbox{ as }\,t\rightarrow\infty,
\end{equation}
where $B_{t,R}=\{x: \abs{x}<\sqrt{t}R\}$. Similarly, for any
$r<\infty$ we obtain
\begin{equation}\label{CKL320-Dec4}
t^{1-\frac{1}{r}}\norm{\bke{\omega(\cdot,t)-\tilde{\omega}(\cdot,t)}}_{L^{r}(B_{t,R})}
\,\,\longrightarrow\,\, 0\qquad \mbox{ as }\,t\rightarrow\infty.
\end{equation}
Since we also have a convergence of $\nabla c$ to
$\nabla\tilde{c}=0$, we can see that
\begin{equation}\label{CKL200-Dec13}
t^{\frac{1}{2}}\norm{\nabla
c(\cdot,t)}_{L^{\infty}(B_{t,R})}\,\,\longrightarrow\,\, 0\qquad
\mbox{ as }\,t\rightarrow\infty.
\end{equation}
Here the point is that the decay estimates are independent of $k$.
Since $\tilde{n}(x,t)=m\Gamma(x,t)$ and
$\tilde{\omega}(x,t)=\gamma\Gamma(x,t)$, we complete the proof.
\end{pfthm1}

\section*{Acknowledgements}
M. Chae's work was partially supported by NRF-2011-0028951. K.
Kang's work was partially supported by NRF-2012R1A1A2001373.  J.
Lee's work was partially supported by NRF-2011-0006697.

\begin{equation*}
\left.
\begin{array}{cc}
{\mbox{Myeongju Chae}}\qquad&\qquad {\mbox{Kyungkeun Kang}}\\
{\mbox{Department of Applied Mathematics }}\qquad&\qquad
 {\mbox{Department of Mathematics}} \\
{\mbox{Hankyong National University
}}\qquad&\qquad{\mbox{Yonsei University}}\\
{\mbox{Ansung, Republic of Korea}}\qquad&\qquad{\mbox{Seoul, Republic of Korea}}\\
{\mbox{mchae@hknu.ac.kr }}\qquad&\qquad {\mbox{kkang@yonsei.ac.kr }}
\end{array}\right.
\end{equation*}
\begin{equation*}
\left.
\begin{array}{c}
{\mbox{Jihoon Lee}}\\
{\mbox{Department of Mathematics }}\\
{\mbox{Chung-Ang University}}\\
{\mbox{Seoul, Republic of Korea}}\\
{\mbox{jhleepde@cau.ac.kr }}
\end{array}\right.
\end{equation*}

\end{document}